\documentclass[12pt]{article}
\usepackage{color}
\usepackage{amsmath}
\usepackage{amsfonts,amssymb}
\usepackage{extarrows}
\usepackage{geometry}
\usepackage{authblk} 
\usepackage{hyperref} 
\usepackage{mathrsfs}
\usepackage[utf8]{inputenc}
\def\0{\emptyset}

\newtheorem{theorem}{Theorem}[section]
\newtheorem{definition}[theorem]{Definition}
\newtheorem{lemma}[theorem]{Lemma}
\newtheorem{claim}[theorem]{Claim}
\newtheorem{observation}[theorem]{Observation}
\newtheorem{cor}[theorem]{Corollary}
\newtheorem{conjecture}[theorem]{Conjecture}
\newtheorem{proposition}[theorem]{Proposition}

\newenvironment{proof}{\noindent\textit{Proof.}}{\hfill $\square$\par}
\usepackage{multicol}
\usepackage{diagbox}
\usepackage{makecell}
\usepackage{multirow}
\usepackage{graphicx}

\usepackage{mathtools} 

\usepackage{enumerate}
\usepackage{enumitem}

\usepackage{
	amsmath,			
	amssymb,			
	enumerate,		    
	graphicx,			
	lastpage,			
	multicol,			
	multirow,			
	pifont,			    
}

\usepackage[numbers]{natbib}

\begin{document}


\title{Linear recoloring diameter of degenerate chordal graphs and bounded treewidth graphs}
\author{

    {\small\bf Yichen Wang}\thanks{email:  wangyich22@mails.tsinghua.edu.cn},\quad
    {\small\bf Mei Lu}\thanks{email: lumei@tsinghua.edu.cn}\\
    {\small Department of Mathematical Sciences, Tsinghua University, Beijing 100084, China.}\\
}

\date{}

\maketitle\baselineskip 16.3pt

\begin{abstract}
	Let $G$ be a graph on $n$ vertices and $t$ an integer. The reconfiguration graph of $G$, denoted by $R_t(G)$, consists of all $t$-colorings of $G$ and two $t$-colorings are adjacent if they differ on exactly one vertex.
    The $t$-recoloring diameter of $G$ is the diameter of $R_t(G)$.
    For a $d$-degenerate graph $G$, $R_t(G)$ is connected when $t \ge d+2$~(Dyer, Flaxman, Frieze and Vigoda, 2006).
    Furthermore, the $t$-recoloring diameter is $O(n^2)$ when $t \ge 3(d+1)/2$~(Bousquet and Heinrich, 2022), and it is $O(n)$ when $t \ge 2d+2$~(Bousquet and Perarnau, 2016).
    For a $d$-degenerate and chordal graph $G$, the $t$-recoloring diameter of $G$ is $O(n^2)$ when $t \ge d+2$~(Bonamy, Johnson, Lignos, Patel and Paulusma, 2014).

    If $G$ is a graph of treewidth at most $k$, then $G$ is also $k$-degenerate, and the previous results hold.
    Moreover, when $t \ge k+2$, the $t$-recoloring diameter is $O(n^2)$~(Bonamy and Bousquet, 2013).
    When $k=2$, the $t$-recoloring diameter of $G$ is linear when $t \ge 5$~(Bartier, Bousquet and Heinrich, 2021) and the result is tight.

    In this paper, we prove that if $G$ is $d$-degenerate and chordal, then the $t$-recoloring diameter of $G$ is $O(n)$ when $t \ge 2d+1$.
    Moreover, if the treewidth of $G$ is at most $k$, then the $t$-recoloring diameter is $O(n)$ when $t \ge 2k+1$.
    This result is a generalization of the previous results on graphs of treewidth at most two.
\end{abstract}


\vskip.3cm

\section{Introduction}

Reconfiguration problems involve finding step-by-step transformations between two feasible solutions such that all intermediate results are also feasible. Let $G$ be a graph on $n$ vertices and $t$ an integer. A \textbf{(proper) $t$-coloring} of $G$ is a function $f: V(G) \rightarrow \{1,\ldots,t\}$ such that for every edge $xy$, $f(x) \neq f(y)$. In this paper, we always discuss proper colorings, so we omit the word ``proper'' for convenience.

We define the configuration graph of $G$ by $R_t(G)$, whose vertices are $t$-colorings of $G$. Two $t$-colorings are \textbf{adjacent} if they differ on exactly one vertex.
We say the graph $G$ is \textbf{$t$-mixing} if $R_t(G)$ is connected. The \textbf{$t$-recoloring diameter} of $G$ is the diameter of $R_t(G)$ denoted by $diam(R_t(G))$. In other words, the $t$-recoloring diameter is the minimum $D$ for which any $t$-coloring of $G$ can be transformed into any other through a sequence of adjacent $t$-coloring of length at most $D$. 
In the following, when we say transforming $\alpha$ to $\beta$, we mean transforming the $t$-coloring $\alpha$ to the $t$-coloring $\beta$ by recoloring one vertex at a time such that all intermediate colorings are also $t$-colorings.
When $G$ is not $t$-mixing~($R_t(G)$ is disconnected), we denote the $t$-recoloring diameter by $+\infty$.
We say the $t$-recoloring diameter of $G$ is at most quadratic if $diam(R_t(G)) = O(n^2)$ and at most linear if $diam(R_t(G)) = O(n)$.
Much previous researches focused on determining the asymptotics of $diam(R_t(G))$~(Linear, quadratic, or even larger) for various graph classes. 
We are also interested in finding the phase transition of $t$ when the $t$-recoloring diameter changes from infinite to finite, from finite to at most quadratic, and from at most quadratic to at most linear.
Note that linear $t$-recoloring diameter is the best we can obtain since each vertex has to be recolored at least once in some cases, for example, every vertex is colored with a different color in two colorings.

The graphs with bounded maximum degree have been extensively studied.
Let $G$ be a graph on $n$ vertices with maximum degree $\Delta(G)$.
Note that, when $t \le \Delta(G) + 1$, there exists a graph $G$ such that $R_t(G)$ is disconnected~(for example, $K_{t+1}$).
Jerrum~\cite{jerrum1995very} proved that if $t \ge \Delta(G) + 2$, then $G$ is $t$-mixing.
Cereceda~\cite{cereceda2007mixing} proved that $diam(R_t(G))= O(n\Delta) = O(n^2)$ when $t \ge \Delta(G) + 2$.
Cambie, Cames van Batenburg and Cranston~\cite{CAMBIE2024103798} studied the exact value of $diam(R_t(G))$ rather than asymptotics, and proved $diam(R_t(G)) \le n + 2\mu(G) \le 2n$ when $t \ge \Delta(G) + 2$, where $\mu(G)$ is the matching number of $G$.


The graph class of $d$-degenerate graphs~(the definition is given in Section~\ref{sec: preliminary}) has also been widely studied.
Also, when $t \le d+1$, $K_{t+1}$ is a $d$-degenerate graph such that $R_t(K_{t+1})$ is disconnected.
Let $G$ be a $d$-degenerate graph on $n$ vertices, Dyer, Flaxman, Frieze and Vigoda~\cite{Dyer2006} proved that $G$ is $t$-mixing when $t \ge d+2$.
Cereceda~\cite{cereceda2007mixing} conjectured the following.
\begin{conjecture}[Cereceda~\cite{cereceda2007mixing}, Conjecture 5.21]\label{conj: cereceda}
    Let $G$ be a $d$-degenerate graph.
    The $t$-recoloring diameter of $G$ is at most quadratic when $t \ge d+2$.
\end{conjecture}
Cereceda~\cite{cereceda2007mixing} confirmed the conjecture when $t \ge 2d+1$.
Bonamy, Johnson, Lignos, Patel and Paulusma~\cite{bonamy2014reconfiguration} proved the conjecture for trees~($d=1$).
Bousquet and Heinrich~\cite{bousquet2022polynomial} improved the conclusion that when $t \ge \frac{3}{2}(d+1)$, the $t$-recoloring diameter of $G$ is at most quadratic.
Bousquet and Perarnau~\cite{bousquet2016fast} proved when $t \ge 2d+2$, the $t$-recoloring diameter is at most linear.

A graph $G$ is \textbf{chordal} if every cycle of length at least four has a chord.
We say $G$ is $r$-colorable, if its chromatic number $\chi(G) $ is at most $r$.
Bonamy, Johnson, Lignos, Patel and Paulusma~\cite{bonamy2014reconfiguration} proved if $G$ is chordal and $r$-colorable, then the $t$-recoloring diameter of $G$ is at most quadratic when $t \ge r+1$.
Since a $d$-degenerate graph is $(d+1)$-colorable, we have that if $G$ is $d$-degenerate and chordal, then the $t$-recoloring diameter is at most quadratic when $t \ge d+2$ which proves the conjecture of Cereceda~under the condition of a chordal graph.
The condition of the chordal graph can not be removed since Bonamy and Bousquet~\cite{bonamy2018recoloring} showed for each $t \ge 3$, there exists a bipartite~($2$-colorable) graph $G$ such that $R_t(G)$ is disconnected.
If we treat the maximum degree as constant, Bousquet and Bartier~\cite{bousquet2019linear} proved for a $d$-degenerate and chordal graph $G$, its $t$-recoloring diameter is $O(f(\Delta)n)$ when $t \ge d+4$, where $\Delta$ is the maximum degree.

Graphs with forbidden subgraphs, and in particular $P_\ell$-free graphs have also been studied.
For $P_4$-free graphs, Bonamy and Bousquet~\cite{bonamy2018recoloring} proved for each $r \ge 1$, if $G$ is $r$-colorable and $P_4$-free, then $G$ is $t$-mixing when $t \ge r+1$.
For $P_6$-free graphs, Bonamy and Bousquet~\cite{bonamy2018recoloring} also proved for each $t \ge 3$, there exists a $2$-colorable $P_6$-free graph $G$ such that $R_t(G)$ is disconnected.
Lei, Ma, Miao, Shi and Wang~\cite{2024arXiv240919368L} proved for any $t \ge r+1 \ge 3$, there exists a $r$-colorable $P_6$-free graph $G$ such that $R_t(G)$ is disconnected.
It means that for every $r \ge 2$, there is no general $T=T(r)$ such that when $t \ge T$, every $P_6$-free $r$-colorable graph $G$ is $t$-mixing.
Then the only remaining case is $P_5$-free graphs.
Bonamy, Johnson, Lignos, Patel and Paulusma~\cite{bonamy2014reconfiguration} proved if $G$ is $P_5$-free and $2$-colorable, then $G$ is $t$-mixing when $t \ge 3$.
Feghali and Fiala~\cite{feghali2020reconfiguration} proved if $G$ is $(P_5, C_5, \overline{P_5})$-free and $3$-colorable, then $G$ is $t$-mixing when $t \ge 4$.
Here $\overline{P_5}$ is the complement graph of $P_5$.
Feghali and Merkel~\cite{feghali2022mixing} proved for every positive integer $p$, there exists a $7p$-colorable $P_5$-free graph such that $R_{8p}(G)$ is disconnected.
Lei, Ma, Miao, Shi and Wang~\cite{2024arXiv240919368L} proved for a $3$-colorable $P_5$-free graph $G$, $R_t(G)$ is connected when $t \ge 4$.
Lei, Ma, Miao, Shi and Wang~\cite{2024arXiv240919368L} also proved for each $r \ge 4$ and $r+1 \le t \le \binom{r}{2}$, there exists a $r$-colorable $P_5$-free graph $G$ such that $R_t(G)$ is disconnected.
The reconfiguration problem for $P_5$-free graphs is still open now.

Finally, graphs with bounded treewidth~(the definition of treewidth is given in Section~\ref{sec: preliminary}) have also been studied.
Treewidth is a stronger parameter than degeneracy.
It is known that a graph of treewidth at most $k$ must be $k$-degenerate.
For a graph $G$ of treewidth at most $k$, Bonamy and Bousquet~\cite{bonamy2013recoloring} proved when $t \ge k+2$, the $t$-recoloring diameter is at most quadratic which proves the conjecture of Cereceda under the condition of treewidth.
When $k=2$, Bartier, Bousquet and Heinrich~\cite{bartier2021recoloring} proved that when $t \ge 5$~(which equals $2k+1$), the $t$-recoloring diameter is at most linear\footnote{The conclusion also holds for $2$-degenerate and chordal graphs. The statement is not explicitly stated in their paper, but can be derived from the proof of Lemma~4 in~\cite{bartier2021recoloring}.}.
It is known that a graph is outerplanar if and only if it has treewidth at most two.
Bonamy, Johnson, Lignos, Patel and Paulusma~\cite{bonamy2014reconfiguration} proved the $4$-recoloring diameter for outerplanar graphs is quadratic.
It completely characterizes the recoloring diameter of outerplanar graphs.

In this paper, we are interested in that, for chordal and $d$-degenerate graphs, when the $t$-recoloring diameter is at most linear. 
We have the following main theorem which can be viewed as a generalization of results in~\cite{bartier2021recoloring}.

\begin{theorem}\label{thm: 2k+1 linear}
    Let $G$ be a $d$-degenerate and chordal graph and $t \ge 2d+1$.
    Given any two $t$-colorings $\alpha, \beta$ of $G$, we can transform $\alpha$ into $\beta$ by recoloring each vertex at most $c$ times, where $c = c(d)$ is a fixed constant only depending on $d$.
\end{theorem}

Using the techniques in~\cite{bartier2021recoloring} and combining Theorem~\ref{thm: 2k+1 linear}, we have the following corollary for graphs of bounded treewidth.

\begin{cor}\label{cor: 2k+1 linear}
    Let $G$ be a graph of treewidth at most $k$ and $t \ge 2k+1$.
    Given any two $t$-colorings $\alpha, \beta$ of $G$, we can transform $\alpha$ into $\beta$ by recoloring each vertex at most $c$ times, where $c = c(k)$ is a fixed constant only depending on $k$.
\end{cor}

\begin{table}[t]
    \centering
    \begin{tabular}{|l|l|l|l|} 
    \hline 
    \multirow{2}{*}{Class} & \multicolumn{3}{c|}{Recoloring diameter}\\ \cline{2-4}
     & $t$-mixing      & at most quadratic & at most linear \\ \hline 
    $d$-degenerate            & $t \ge d+2$~\cite{Dyer2006}~(tight) &  $t \ge \frac{3}{2}(d+1)$~\cite{bousquet2022polynomial}         &  $t \ge 2d+2$~\cite{bousquet2016fast}      \\ \hline  
    \makecell[l]{$d$-degenerate\\ and chordal}           & $t \ge d+2$~\cite{Dyer2006}~(tight) &  $t \ge d+2$~\cite{bonamy2014reconfiguration}~(tight)         &  $t \ge 2d+1$~(Theorem~\ref{thm: 2k+1 linear})      \\ \hline  
    \makecell[l]{treewidth $\le k$\\($k \ge 2$)}     & $t \ge k+2$~\cite{Dyer2006}~(tight) &  $t \ge k+2$~\cite{bonamy2013recoloring}~(tight)  & $t \ge 2k+1$~(Theorem~\ref{thm: 2k+1 linear})        \\ \hline 
    treewidth $\le 2$     & $t \ge 4$~\cite{Dyer2006}~(tight) &  $t \ge 4$~\cite{bonamy2013recoloring}~(tight) & $t \ge 5$~\cite{bartier2021recoloring}~(tight:~\cite{bonamy2014reconfiguration})       \\ \hline 
    \end{tabular}
    \caption{
        Research on recoloring diameter for degenerate graphs, chordal graphs, and bounded treewidth graphs. `Tight' means the bound on $t$ cannot be improved.
        }\label{tab: recoloring diameter}
\end{table}

Table~\ref{tab: recoloring diameter} shows the current best bounds of $t$ about the $t$-recoloring diameter.

The list coloring version and correspondence coloring version of recoloring diameter are also studied, the readers may refer to~\cite{CAMBIE2024103798, 2025arXiv250508020C}.
There are also studies focusing on specific graphs, for example, Cambie, Cames van Batenburg and Cranston~\cite{2024arXiv241219695C} studied the exact value of the recoloring diameter of $K_{p,q}$.

This paper is organized as follows.
In Section~\ref{sec: preliminary}, we introduce basic definitions and notations.
In Section~\ref{sec: best choice algorithm}, we introduce the best choice algorithm which is widely used in $t$-recoloring diameter research.
In Section~\ref{sec: main}, we prove Theorem~\ref{thm: 2k+1 linear} and Corollary~\ref{cor: 2k+1 linear}.
In Section~\ref{sec: discussion}, we conclude the paper and introduce some future questions.

\section{Preliminary}\label{sec: preliminary}

We first give definitions involving treewidth.

\begin{definition}
	A \textbf{tree-decomposition} of a graph $G = (V, E)$ is a pair $(X, T)$,
	where $T(I, F)$ is a tree with vertex set $I$ and edge set $F$, and $X = \{X_i \mid i \in I\}$ is a family of subsets of $V$, one for each node of $T$, such that:
	\begin{enumerate}[label=(\arabic*)]
		\item $\bigcup \limits_{i \in I}X_i = V$;
		\item for each edge $uv \in E$, there exists an $i \in I$ such that $u,v \in X_i$;
		\item for each vertex $v \in V$, the induced subgraph of $T$ containing all nodes $i$ such that $v \in X_i$ is a subtree. That is, $T(v) \coloneqq T[\{i \in I \mid v \in X_i\}]$ is a subtree.
	\end{enumerate}
\end{definition}

For any node $i$ of $T$, there is a vertex set $X_i$, called \textbf{bag}, that corresponds to $i$.
In the following, we do not distinguish $i$ and $X_i$ for convenience.
The width of a tree-decomposition $(X, T)$ is $\max \limits_{i\in I} |X_i| - 1$.
The treewidth of a graph $G$ is the minimum treewidth over all possible tree-decompositions of $G$.


A graph is \textbf{$d$-degenerate} if every subgraph has a vertex of degree at most $d$.
Or equivalently, there exists an ordering $v_1, \ldots, v_n$ of the vertices of $G$ such that for any $i \ge 1$, $v_i$ has at most $d$ neighbors in $\{v_1, \ldots, v_{i-1}\}$, that is, $|N_G(v_i) \cap \{v_1, \ldots, v_{i-1}\}| \le d , 1 \le i \le n$ where $N_G(v) = \{u \mid uv \in E(G)\}$ and we ignore the subscript when there is no ambiguity.
Given an ordering $v_1, \ldots, v_n$ of the vertices of $G$, we write $V_i = \{v_1, \ldots, v_i\}$, $G_i = G[V_i]$ which is the subgraph of $G$ induced by $V_i$, and $N^-(v_i) = N(v_i) \cap V_{i-1}$.
In this paper, we always discuss graphs with an ordering of vertices.
When we say $N^-(v)$, we mean $N^-(v_i)$ where $v = v_i$ in the ordering.

It is known that a chordal graph $G$ admits a \textbf{perfect elimination ordering} which is an ordering $v_1, \ldots, v_n$ of the vertices of $G$ such that for any $v_i, 1 \le i \le n$, $N^-(v_i)$ is a clique.
Moreover, if $G$ is a $d$-degenerate and chordal graph, then the clique number of $G$ is at most $d+1$ by the definition of degenerating.
Then in the perfect elimination ordering of $G$, $N^-(v_i)$ has size at most $d$ for any $1 \le i \le n$.

In the following, we will associate a graph with an ordering of its vertices.
When we discuss a $d$-degenerate and chordal graph $G$, we always associate it with its perfect elimination ordering.




\section{Best choice algorithm}\label{sec: best choice algorithm}

In this section, we introduce the \textbf{best choice algorithm} which is widely used in $t$-recoloring diameter research~\cite{bartier2021recoloring,bousquet2016fast}. Our notation is also from~\cite{bartier2021recoloring}.
We shall make a detailed analysis of the algorithm to obtain our main theorem.


In the recoloring process, since in each step we only recolor one vertex, we can characterize the recoloring process by a sequence of recoloring pairs.
Given a graph $G$ and a coloring $\alpha_0$ of $G$, a \textbf{recoloring sequence} $\mathcal{S}$ is a sequence $s_1s_2 \ldots s_m$ where each $s_i$ is a pair $(v_i,c_i)$ describing the vertex $v_i$ and its new color $c_i$ in the $i$-th step.
For a recoloring pair $s=(v,c)$, we say \textbf{$s$ recolors $v$ to $c$} and $c$ is the new color of $v$.
Then we can get a coloring sequence from $\mathcal{S}$ by letting $\alpha_i, 1 \le i \le m$  the coloring obtained from $\alpha_{i-1}$ by changing the color of $v_i$ to $c_i$.
All the intermediate colorings described by this sequence must be proper.
Then we say the $\mathcal{S}$ is a \textbf{recoloring sequence} from $\alpha_0$ to $\alpha_m$.

When we say the coloring of $G$ before step $i$, we mean the coloring $\alpha_{i-1}$.
Sometimes for convenience, we also say $\alpha_{i-1}$ the coloring at the moment right before step $i$.

Let $X$ be a vertex subset of $G$. The restriction of $\mathcal{S}$ to $X$, denoted by $\mathcal{S}_{|X}$, is the recoloring sequence obtained from $\mathcal{S}$ by selecting recoloring of vertices in $X$~(that is, $v_i \in X$).
In particular, if $X=\{v\}$, then $\mathcal{S}_{|v}$ denotes the sequence of colors taken by the vertex $v$ in $\mathcal{S}$.
Let $|\mathcal{S}|$ denotes the length of the sequence $\mathcal{S}$. Then $|\mathcal{S}_{|v}|$ is the number of recoloring of $v$ in $\mathcal{S}$.

A \textbf{pattern} is a sequence of vertices $v_0 v_1 \ldots v_p$.
We say the pattern appears in $\mathcal{S}$ if there is an index $i$ such that for all integers $0 \le j \le p$, $s_{i+j}$ recolors the vertex $v_j$. And we say the pattern appears $N$ times if there are $N$ such indexes. Note that if a pattern appears multiple times, its copies may overlap.

The idea of the best choice algorithm is to extend a recoloring sequence from a subgraph to the whole graph.
Consider a graph $G$ and two colorings $\alpha, \beta$ of $G$ and a vertex $u$ of $G$.
Let $\alpha_{|G-u}$ and $\beta_{|G-u}$ be the restrictions of $\alpha$ and $\beta$ to $G-u$ and $\mathcal{S}$ a recoloring sequence from $\alpha_{|G-u}$ to $\beta_{|G-u}$.
Let us now explain how we can extend this recoloring sequence $\mathcal{S}$ to the whole graph $G$.

Let $t_1, \ldots, t_\ell$ be all the steps where neighbor of $u$ is recolored in $\mathcal{S}$.
That is, $\mathcal{S}_{|N(u)} = s_1's_2'\ldots s_\ell'$ where $s_i' = s_{t_i}$.
Assume $s_i' = s_{t_i} = (v_{t_i}, c_{t_i}), 1 \le i \le \ell$.
Here $c_{t_i}$ is the new color assigned to the recolored neighbor at step $t_i$. If for some $t$, $c_{t}$ is the same as the current color of $u$, then we have to recolor $u$ before step $t$ to avoid conflict.
The \textbf{best choice} for $u$ at step $t \in \{t_1, \ldots, t_\ell\}$ is
\begin{itemize}
    \item the color $\beta(u)$ if $\beta(u)$ is distinct from $c_{t}, \ldots, c_{t_\ell}$;
    \item or any valid choice for $u$ at step $t$ which is distinct from $\{c_{t_i} \mid t_i \ge t\}$ otherwise;
    \item or the valid choice for $u$ at step $t$ that appears the latest in the sequence ${(c_{t_i})}_{t_i \ge t}$ otherwise.
\end{itemize}

The three rules are from~\cite{bartier2021recoloring}.
The first and second rules are explicit.
We explain more about the third rule when the first two rules fail.
When $u$ has to be recolored, the valid colors are $[t] \coloneqq \{1,2,\ldots, t\}$ but the colors of $N[u] \coloneqq N(u) \cup \{u\}$.
We aim to let the number of recoloring of $u$ as small as possible.
So we postpone the time $u$ is recolored next time as late as possible.
Therefore, we choose the largest $j$ such that the set of valid colors but ${(c_{t_k})}_{t \le t_k \le t_j}$ is not empty.
When $j$ attains a maximum, only one valid color is remaining.
Then we choose this valid color to be the best choice of $u$, which is the color that appears the latest in the sequence ${(c_{t_i})}_{t_i \ge t}$.

The \textbf{Local Best Choice} for $u$ extends the sequence $\mathcal{S}$ by recoloring only the vertex $u$.
When a neighbor of $u$ (say $v$) is recolored in $\mathcal{S}$ with the current color of $u$, we add a recoloring for $u$ just before this recoloring and recolor $u$ with its best choice. In this case, we say the recoloring of $u$ is \textbf{caused} by $v$.
We do not perform any other recoloring for $u$ except at the very last step to give it color $\beta(u)$ if needed.

Let $G$ be a graph with ordering $v_1,\ldots,v_n$ of its vertices and $\alpha, \beta$ two colorings of $G$.
The Best Choice Recoloring Algorithm runs the Local Best Choice successively on $v_i, 1 \le i \le n$ from $\alpha_{|G_i}$ to $\beta_{|G_i}$.
The \textbf{best choice recoloring sequence} is the recoloring sequence $\mathcal{S}$ obtained from this procedure.
Observe that by the construction, if $\mathcal{S}$ is a best choice recoloring sequence for $G$ and the ordering $v_1, \ldots, v_n$ from $\alpha$ to $\beta$, then for every $i \ge 1$, $\mathcal{S}_{|V_i}$ is a best choice recoloring sequence of $G_i$ from $\alpha_{|G_i}$ to $\beta_{|G_i}$.

We will apply the algorithm to $d$-degenerate and chordal graphs on its perfect elimination ordering.
The properties of the perfect elimination ordering will ensure the algorithm does not fail when $t \ge d+2$.
Note that when we construct a recoloring sequence from $t$-coloring $\alpha$ to $t$-coloring $\beta$ of $G$, we actually find a path in $R_t(G)$ from $\alpha$ to $\beta$.

\section{Proof of Theorem~\ref{thm: 2k+1 linear} and Corollary~\ref{cor: 2k+1 linear}}\label{sec: main}

We first explain how to derive Corollary~\ref{cor: 2k+1 linear} from Theorem~\ref{thm: 2k+1 linear}.
We need the following proposition.

\begin{proposition}[\cite{bartier2021recoloring}]\label{prop: 2k+1 tw k}
    Let $G$ be a graph of treewidth at most $k$, $t \ge 2k+1$, and $\alpha, \beta$ two $(k+1)$-colorings of $G$. Then using colors $[t]$, $\alpha$ can be transformed into $\beta$ by recoloring every vertex at most twice.
\end{proposition}

Note that in Proposition~\ref{prop: 2k+1 tw k}, $\alpha$ and $\beta$ only use $(k+1)$ colors while in the transformation process, we can use $t$ colors.
Then we need the following trick used in~\cite{Dyer2006, bartier2021recoloring}.

Given a graph $G$ of treewidth at most $k$ and a $t$-coloring $\alpha$, let $(X, T)$ be a tree-decomposition of $G$ of width at most $k$. 
We are going to build a $k$-degenerate and chordal graph $G''$ from $G$.
For every bag $B$ of $T$, we merge the vertices of $B$ with the same color in $\alpha$.
That is, if two vertices $u,v \in B$ have the same color in $\alpha$, we identify $u$ and $v$ into a new vertex $w$ where $N_{G'}(w) = N_G(u) \cup N_G(v)$ in the new graph $G'$.
And let $\alpha'$ be the coloring of $G'$ obtained from $\alpha$ by preserving the colors of the merged vertex and other vertices.
We repeat it until we obtain a graph $G'$ with coloring $\alpha'$ such that no bag contains two vertices with the same color.
Note that $tw(G') \le tw(G)$ since we only identify the vertices belonging to the same bags.
By the definition, it is clear that $\alpha'$ is a proper coloring of $G'$.
Let $G''$ be the graph obtained from $G'$ by adding edges such that every bag forms a clique.
Since vertices in a bag receive distinct colors in $\alpha'$, adding these edges will not create any conflict in $\alpha'$.
We can verify that $G''$ is a $k$-degenerate and chordal graph with a proper coloring $\alpha'$.
Note that every vertex $v$ in $G''$ corresponds to an independent set in $G$ denoted by $\mathcal{I}(v)$.
Given a coloring $\gamma'$ of $G''$, we can obtain a coloring $\gamma$ of $G$ where for each $v$, every vertex in $\mathcal{I}(v)$ has the color $\gamma'(v)$. 
Since each $\mathcal{I}(v)$ is an independent set in $G$, $\gamma$ is a proper coloring of $G$.
Moreover, if we can transform $\alpha'$ into a coloring $\gamma'$ of $G''$ by recoloring every vertex at most $c$ times, then we can transform $\alpha$ into coloring $\gamma$ of $G$ by recoloring every vertex at most $c$ times where $\alpha$ and $\gamma$ are the corresponding coloring of $G$ from $\alpha'$ and $\gamma'$, respectively.

\noindent
\textbf{Proof of Corollary~\ref{cor: 2k+1 linear} using Theorem~\ref{thm: 2k+1 linear}}

Let $G$ be a graph of treewidth at most $k$ and $\alpha, \beta$ two $t$-colorings of $G$ with $t \ge 2k+1$.
Let $G''$ be the $k$-degenerate and chordal graph obtained from $G$ as described before.
Let $\alpha'$ and $\beta'$ be the coloring of $G''$ obtained from $\alpha$ and $\beta$, respectively.
By Theorem~\ref{thm: 2k+1 linear}, we can transform $\alpha'$~(resp. $\beta'$) into a $(k+1)$-coloring $\gamma_1'$~(resp. $\gamma_2'$) of $G''$ by recoloring each vertex at most $c$ times.
Assume $\gamma_1'$ and $\gamma_2'$ correspond to $(k+1)$-coloring $\gamma_1$ and $\gamma_2$ of $G$, respectively.
Therefore, in $G$, we can transform $\alpha$~(resp. $\beta$) into $\gamma_1$~(resp. $\gamma_2$) by recoloring each vertex at most $c$ times.

From Proposition~\ref{prop: 2k+1 tw k}, we have that $\gamma_1$ can be transformed into $\gamma_2$ by recoloring each vertex at most twice.
As a conclusion, we can transform $\alpha$ into $\beta$ by recoloring each vertex at most $2c+2$ times.
\hfill $\square$\par

The rest of this section is devoted to proving Theorem~\ref{thm: 2k+1 linear}.
In the following, let $G$ be a $d$-degenerate and chordal graph, and $v_1, \ldots, v_n$ be a perfect elimination ordering of $G$.
And let $\alpha$ and $\beta$ be two $t$-colorings of $G$ with $t \ge 2d+1$.
Since the conclusion holds when $t \ge 2d+2$ from the result of $d$-degenerate graphs, we just need to prove when $t = 2d+1$.
And, since the situation when $d=2$ is proved~\cite{bartier2021recoloring}, we assume $d \ge 3$.
We use $[t]= \{ 1, \ldots, t\}$ to denote the set of colors.
Assume $\mathcal{S} = s_1\ldots s_m$ is a best choice recoloring sequence for this ordering between two $\alpha$ and $\beta$.

Let $v$ be any vertex of $G$, it is sufficient to prove $v$ is recolored at most $c$ times in $\mathcal{S}$.
Then constant $c=c(d)$ will be determined later.
Assume $N^-(v) = {\{u_i\}}_{1 \le i \le p}$ where $p \le d$.
Except for the last recoloring of $v$ in which we recolor $v$ to $\beta(v)$, every recoloring of $v$ must be caused by some $u_i$.
According to the best choice rules for $v$, we have the following observations.

\begin{observation}\label{obs: tight recoloring}
    Let $v$ be any vertex of $G$ and $\mathcal{S}$ be the best choice recoloring sequence of $G$.
    Let $w_i \in N^-(v), 1 \le i \le \ell$~(repetition is allowed), where $\ell \le d-1$.
    Then the pattern $vv$ never occurs in $\mathcal{S}_{|N^-[v]}$ and the pattern $vw_1\ldots w_{\ell}v$ occurs at most once, when the second recoloring of $v$ is the last recoloring of $v$ in $\mathcal{S}_{|N^-[v]}$.
\end{observation}

\begin{proof}
    If the recoloring of $v$ is not the last one, then it must be caused by a recoloring of a neighbor of $v$. Hence, $vv$ never appears.
    If pattern $vw_1\ldots w_{\ell}v$ appears and the second recoloring of $v$ is not the last recoloring of $v$, then the second recoloring of $v$ is caused by some $w' \in N^-(v)$.
    Now consider the first recoloring of $v$ in this pattern.
    Obviously, it does not satisfy the first two rules of the best choice.
    So the recoloring rule applied is the third one. However, when $v$ is recolored first in this pattern, there are at least $d$ valid colors~($[2d+1]$ minus the colors of $N^-[v]$).
    The recoloring of $w_1$ uses the color of $v$ since the recoloring of $v$ is caused by $w_1$.
    It means the recoloring of $w_1$ does not use a valid color.
    The rest recolorings of $w_2, \ldots, w_{\ell}$ use up to $d-2$ valid colors.
    There exists at least two valid colors $c_1, c_2$ not used in $N^-[v]$ and the recolorings of $w_1, \ldots, w_{\ell}$.
    Let $c_1$ be the new color of $v$ after the first recoloring.
    Now consider the recoloring of $w'$, which causes the second recoloring of $v$. It means the recoloring of $w'$ use the color $c_1$.
    
    By the third rule for the first recoloring of $v$, it should never choose $c_1$ as there is still another valid color $c_2$ remaining.
    This is a contradiction.
\end{proof}

\vspace*{.2cm}
If a recoloring of $v$ in $\mathcal{S}_{|N^-[v]}$ satisfies that there are exactly $d$ recoloring of vertices from $N^-(v)$ between the recoloring of $v$ and the next recoloring of $v$, then we say the recoloring of $v$ is \textbf{tight}.
Note that the definition of tight recoloring only depends on the sub-sequence $\mathcal{S}_{|N^-[v]}$. 
We have the following observation about tight recoloring.

\begin{observation}\label{obs: tight recoloring property 2}
    Let $v$ be any vertex of $G$ and $\mathcal{S}$ be the best choice recoloring sequence of $G$.
    Suppose $N^-(v) = \{u_1, \ldots, u_d\}$, and there is a tight recoloring of $v$, say in the pattern $vw_1\ldots w_d v$, where $w_i \in N^-(v)$~(repetition is allowed).
Let us denote by $c_1,\ldots,c_d$ the colors of $u_1\ldots,u_d$ before the first recoloring of $v$ and by $c_0$ the color of $v$ before this recoloring.
Let $c_i'$ be the new color of $w_i$ for $1 \le i \le d$. Then we have $[2d+1] = \{c_0, c_1, \ldots, c_{d}, c_1', \ldots, c_d'\}$ unless the last recoloring of $v$ is the last step in $\mathcal{S}_{|N^-[v]}$.
\end{observation}

\begin{proof}
    Similar to the proof of Observation~\ref{obs: tight recoloring}, the conclusion follows from the third rule of best choice for the first recoloring of $v$.
\end{proof}

Assume the best choice recoloring sequence restricted to $N^-[v]$ is $\mathcal{S}_{|N^-[v]} = s_1' s_2' \ldots s_{m}'$.
We say the step $i$ is \textbf{saved} for $v$ if at the step $s_i'$, $w \in N^-(v)$ is the vertex being recolored and one of the following holds:

\begin{enumerate}
    \item $v$ is not recolored at step $1, \ldots, i$;
    \item $v$ is not recolored at step $i, \ldots, m$;
    \item $v$ is not recolored at all in the $d$ steps before $s_i'$ in $\mathcal{S}_{|N^-[v]}$.
\end{enumerate}

Lemma~\ref{lemma: save inequality} shows an inequality involving in number of saved recoloring.

\begin{lemma}\label{lemma: save inequality}
    Let $v$ be any vertex of $G$ and $\mathcal{S}$ be the best choice recoloring sequence of $G$.
    If $r$ is the number of saved recoloring for $v$, then the following inequality holds:
    \[
    |\mathcal{S}_{|v} | \le 1 + \left\lceil \frac{\sum_{u \in N^-(v)} |\mathcal{S}_{|u} | - r}{d} \right\rceil.
    \]
\end{lemma}

\begin{proof}
    We prove the lemma by induction on $\kappa(\mathcal{S}) \coloneqq \sum_{u \in N^-(v)} |\mathcal{S}_{|u} |$. 
    Note that $\kappa(\mathcal{S}) \ge r$ by the definition of saved recoloring.
    If $\kappa(\mathcal{S}) = 0$, note that except the last recoloring of $v$, every recoloring of $v$ must be caused by some neighbor of $v$. Then $v$ is recolored at most once and the conclusion holds.
    In the following, we assume $\kappa(\mathcal{S}) \ge 1$, and the conclusion holds for any vertex $\mathcal{S}'$ with $\kappa(\mathcal{S}') < \kappa(\mathcal{S})$.

    If $v$ is recolored at most once, then the conclusion holds.
    Otherwise, $v$ is recolored at least twice. If $v$ is not the first recolored vertex of $\mathcal{S}_{|N^-[v]}$, then the first recoloring is saved.
    If we consider the sub-sequence $\mathcal{S}'$ starting from the second recoloring of $\mathcal{S}_{|N^-[v]}$.
    Then $\kappa(\mathcal{S}') = \kappa(\mathcal{S}) -1$, and the number of saved recoloring of $v$ in $\mathcal{S}'$ is $r-1$.
    By induction, $v$ is recolored at most $1 + \lceil ((\kappa-1) - (r-1))/d \rceil$ times which is exactly what we need.

    If $v$ is the first vertex recolored in $\mathcal{S}_{|N^-[v]}$. Write $\mathcal{S}_{|N^-[v]} = s_1' \ldots s_{\ell}'$. 
    We know $v$ is recolored at $s_1'$.
    Since $v$ is recolored at least twice, assume the next time $v$ is recolored is $s_p'$.
    By Observation~\ref{obs: tight recoloring}, $p \ge d+2$ unless $p = \ell$. If $p = \ell \le d + 1$, then $ 1 \le \kappa(\mathcal{S}) \le d-1$ and $r=0$. We can verify the inequality holds.

    Otherwise, we have $p \ge d+2$. In this case let the sub-sequence $\mathcal{S}'$ start from $s_{d+2}'$, then by induction, $v$ is recolored in $\mathcal{S}'$ at most $1 + \left\lceil \frac{\kappa(\mathcal{S})-d -r}{d} \right\rceil$ times. Adding the recoloring of $v$ at $s_1'$, we can verify the inequality.
\end{proof}

Note that if every vertex in $N^-(v)$ is recolored at most $c$ times, then $v$ is recolored at most $c+1$ times by Lemma~\ref{lemma: save inequality}. 
Moreover, if $|N^{-}(v)| < d$, then $v$ is recolored at most $c$ times.

\begin{lemma}\label{lemma: cause chain save}
    Let $v$ be any vertex of $G$ and $\mathcal{S}$ be the best choice recoloring sequence of $G$.
    Suppose $N^-(v) = \{u_1, \ldots, u_d\}$ and $\mathcal{S}_{|N^-[v]} = s_1' \ldots s_{m}'$.
    Assume $s_{\ell}'$ is the last recoloring of $v$ in $\mathcal{S}_{|N^-[v]}$.
    Assume at $s_j'$~($1 \le j < \ell$), $u_p$ is recolored and is caused by $u_q$ for some $1 \le p,q \le d$.
    Let $i < j$ be the first step before $j$ in which $v$ is recolored~(assume it exists).
    Then either $i + (d+1) \ge \ell$ or the step $i+(d+1)$ is saved for $v$.
    Moreover, let $\gamma$ denote the number of such $j$'s, then there are at least $\gamma - d$ saved recolorings for $v$.
\end{lemma}

\begin{proof}
    If $i + (d+1) < \ell$, from Observation~\ref{obs: tight recoloring}, we have $v$ is not recolored in all $s_{i+1}', \ldots, s_{i+d}'$.
    It is sufficient to prove that $s'_{i+d+1}$ does not recolor $v$ and thus it is a saved recoloring for $v$.
    If $v$ is recolored at $s'_{i+d+1}$, then by the way we choose $i$, we have $i < j < i+d$.
    Let $c_1, \ldots, c_d$ be the colors of $u_1, \ldots, u_d$ before the recoloring of $v$ at $s_i'$ and let $c_0$ be the color of $v$ before $s_i'$.
    Let $c_1', \ldots, c_d'$ be the new colors of $s_{i+1}', \ldots, s_{i+d}'$, respectively.
    By the best choice algorithm, we have $c_1' = c_0$.
    Since $s_j'$ is caused by some $u_q$, then at step $j+1$, $u_q$ is recolored with the color $c_{j+1}'$.
    If $u_q$ has not been recolored between step $i+1$ and step $j$, then $c_{j+1}' = c_p$.
    Otherwise, we have $c_{j+1}' \in \{c_1',\ldots,c_j'\}$.
    In both cases, we have
    \[
        \left|\{c_0, c_1, \ldots, c_d, c_1', \ldots, c_d'\}\right| \le 2d-1.
    \]
    Since $i+d+1 < \ell$, then at $s_{i+d+1}'$, the recoloring of $v$ is caused by some $u' \in N^-(v)$. 
    Which means, $u'$ is going to be recolored with the new color of $v$ at step $i$.
    Since there are still at least two valid colors for $v$ at $s_i'$, it violates the third rule of the best choice of $v$ at $s_i'$. The reason here is similar to the proof of Observation~\ref{obs: tight recoloring}.

    Note that if $i+d+1 > \ell$, then $j+d > \ell$, and any recoloring of vertices in $N^-(v)$ after $s_{\ell}'$ is a saved recoloring for $v$.
    The second statement holds by a similar argument.
\end{proof}

\vspace{.2cm}
The following lemma is the key to bound the number of recoloring of a vertex.

\begin{lemma}\label{lemma: refine sequence}
    Let $x$ be any vertex of $G$ and $\mathcal{S}$ be the best choice recoloring sequence of $G$.
    Assume $N^-(x) = \{y_1, \ldots, y_d\}$ and $N^-(y_1) = \{y_2, \ldots, y_d, w\}$. If $x = v_i$ in the perfect elimination ordering, $c \ge \max_{i' < i}|\mathcal{S}_{|v_{i'}}|$, and $x$ is recolored at least $c-p$ times in $\mathcal{S}$ where $p$ is a non-negative integer~($|\mathcal{S}_x| \ge c-p$), then we have the following properties.
    \noindent
    \begin{enumerate}
        \item The patterns $\{y_1w(z_2\ldots z_{d})y_1 \mid z_j \in \{y_2, \ldots, y_d\}, \forall 2 \le j \le d \}$ appear at least $c - 3d^2(p+2)$ times in $\mathcal{S}_{|N^-[y_1]}$ in total.
        \item The patterns
        \[
            y_1w(z_2\ldots z_d)xy_1w(z_2'\ldots z_d') xy_1w(z_2''\ldots z_d'')xy_1w(z_2'''\ldots z_d''')xy_1,
        \]
        where $z_j, z_j', z_j'',z_j''' \in \{y_2, \ldots, y_d\}, \forall 2 \le j \le d$, appear at least $c-16d^2(p+2)$ times in $\mathcal{S}_{|N^-[y_1]\cup \{x\}}$ in total.
        \item Let $s'$ be a recoloring of $x$, and $s'', s'''$ be the next two recolorings of $x$ after $s'$ in $\mathcal{S}_{|N^-[y_1]\cup \{x\}}$.
        Let $c_1$ denote the color of $x$ before $s'$, $c_2$ the color of $x$ after $s'$, and $c_3$ the color of $x$ after $s''$.
        We say the recoloring $s'$ of $x$ is \textbf{rotating} if at $s'''$, $x$ is recolored back to $c_1$.
        That is, the color of $x$ behaves like $c_1 \xrightarrow[]{s'} c_2 \xrightarrow[]{s''} c_3 \xrightarrow[]{s'''} c_1$.
        Then at most $16d^2(p+2)+1$ recolorings of $x$ are not rotating.
        Moreover, for every $xy_1w$ appears in the patterns in statement 2, the colors of $x, y_1, w$ are the same before and after the occurrence of this pattern.
    \end{enumerate}
\end{lemma}

\begin{proof}
    \textbf{Proof of statement 1.}
    Write $\mathcal{S}_{|N^-[y_1]} = s_1' \ldots s_{\ell}'$.
    From Observation~\ref{obs: tight recoloring}, between each two consecutive recolorings of $y_1$ in $\mathcal{S}_{|N^-[y_1]}$, there are at least $d$ elements unless the later one is the last recoloring of $y_1$.

    We claim that there are at most $(p+2)d-1$ saved recolorings for $x$. Otherwise, from Lemma~\ref{lemma: save inequality}, we have $c-p \le |\mathcal{S}_{|x}| \le 1 + \left\lceil \frac{\sum_{u \in N^-(x)} |\mathcal{S}_{|u}| - (p+2)d}{d} \right\rceil \le c-p-1$, a contradiction.
    Similarly, we have that $y_1$ is recolored at least $c - ((p+2)d-1)$ times.

    If there are $\gamma$ recolorings of $y_1$ caused by some $z \in \{y_2, \ldots, y_d\}$, by Lemma~\ref{lemma: cause chain save}, it will result in $\gamma -d$ saved recolorings for $x$.
    As a result, there are at most $(p+3)d-1$ recolorings of $y_1$ which is not caused by $w$ and furthermore, the pattern $y_1w$ appears at least $c - ((2p+5)d-2)$ times in $\mathcal{S}_{|N^-[y_1]}$. 
    Moreover, we have that at most $((2p+5)d-2)$ recolorings of $w$ that does not cause a recoloring of $y_1$.

    Now, we consider the sequence $\mathcal{S}_{|N^-[y_1]}$. Note that $y_1$ is recolored at least $c-((p+2)d-1)$ times and hence there are at most $d((p+2)d+1)-1$ saved recolorings for $y_1$. 
    Note that every non-tight recoloring of $y_1$ gives a saved recoloring for $y_1$.
    From Observation~\ref{obs: tight recoloring}, between two consecutive recolorings of $y_1$, there are exactly $d$ elements with at most $d((p+2)d+1)$ exceptions~(last one and non-tight ones).
    For every tight recoloring of $y_1$, it is always caused by $w$ with at most $((2p+5)d-2)$ exceptions.
    For a tight recoloring of $y_1$ caused by $w$, $w$ does not appear in the rest $(d-1)$ recolorings with at most $((2p+5)d-2)$ exceptions.
    As a conclusion, the patterns $y_1w(z_2\ldots z_{d})y_1$ appear at least 
    \[
        c - ( (p+2)d-1) - 1 - d((p+2)d+1 ) - 2((2p+5)d-2) \ge c - 3d^2(p+2)
    \]
    times.

    \textbf{Proof of statement 2.}
    By statement 1, we have that $y_1w(z_2\ldots z_d)y_1$ appears at least $c-3d^2(p+2)$ times in $\mathcal{S}_{|N^-[y_1]}$.
    If a recoloring of $y_1$ is the first recoloring in the pattern $y_1w(z_2\ldots z_d)y_1$ in $\mathcal{S}_{|N^-[y_1]}$, we say the recoloring of $y_1$ is \textbf{good}.
    That is, there are at least $c-3d^2(p+2)$ good recolorings of $y_1$ in $\mathcal{S}_{|N^-[y_1]}$.
    Now we consider the recoloring sequence $\mathcal{S}_{|N^-[y_1]\cup \{x\}}$. We can view it as inserting recoloring of $x$ into $\mathcal{S}_{|N^-[y_1]}$. Since $x$ is recolored at least $c-p$ times, there are at most $d(p+2)-1$ saved recoloring for $x$. Hence, at most $d(p+2)$ recoloring of $x$ is not tight.
    In the pattern $y_{1}w(z_2\ldots z_d)y_1$, there must be a recoloring of $x$ inside with at most $d(p+2)+1$ exceptions.
    Therefore, there are at least $c-3d^2(p+2) - 2(d(p+2)+1) \ge c - 4d^2(p+2)$ good recoloring of $y_1$ and there is a tight recoloring of $x$ inserted before the next recoloring of $y_1$.

    Since we have at least $c-4d^2(p+2)$ and at most $c$ such recolorings of $y_1$, there must be at least $c-16d^2(p+2)$ four consecutive such recolorings of $y_1$.
    That is, the pattern
    \[
    y_1w(z_2\ldots z_d)y_1w(z_2'\ldots z_d')y_1w(z_2''\ldots z_d'')y_1w(z_2'''\ldots z_d''')y_1
    \]
    where $z_j, z_j', z_j'',z_j''' \in \{y_2, \ldots, y_d\}, 2 \le j \le d$, appears at least $c-16d^2(p+2)$ times in $\mathcal{S}_{|N^-[y_1]}$. 
    Moreover, between each two consecutive recolorings of $y_1$ in the pattern, there is a tight recoloring of $x$ inserted.
    Here four consecutive such patterns are necessary for the following proof.
    

    Now we claim the recoloring of $x$ must be inserted just before $y_1$.
    Otherwise, suppose the first recoloring of $x$ is inserted before $z_{\ell+1}$ for some $\ell < d$. Since all inserted recolorings of $x$ are tight, the next recolorings of $x$ are inserted exactly before $z_{\ell+1}', z_{\ell+1}''$ and $z_{\ell+1}'''$.
    As a result, we have a (sub)-sequence in $\mathcal{S}_{|N^-[y_1]\cup \{x\}}$ as follows.
    \[
    y_1w(z_2\ldots z_{\ell}) x^{(1)}(z_{\ell+1}\ldots z_d)y_1^{(1)}w(z_2'\ldots z_{\ell}') x^{(2)}(z_{\ell+1}'\ldots z_d') y_1^{(2)}w(z_2''\ldots z_{\ell}'') x^{(3)}(z_{\ell+1}''\ldots z_d'')y_1^{(3)}.
    \]
    Here for convenience, we use $x^{(i)},y^{(i)}, i=1,2,3$, to distinguish different recolorings of $x$ and $y_1$, respectively.
    Let $t_i, i=1,2,3$, be the moment just before the recoloring of $x^{(i)}$.
    Then $y_1^{(i)}, i=1,2,3$, is the first recoloring of $y_1$ after $t_i$.
    We may assume $y_1^{(3)}$ is not the last recoloring of $y_1$ in $\mathcal{S}_{|N^-[y_1]\cup \{x\}}$ since it only happens at most once.
    Let $c_{v}^{t_i}$ be the color of vertex $v$ at moment $t_i$.
    Without loss of generality, we assume $c_{x}^{t_1} =1$ and $c_{y_i}^{t_1} = i+1, 1 \le i \le d$.
    See Figure~\ref{fig: sequence} for an example.


    The recoloring of $y_1^{(1)}$ is caused by $w$, hence,  $c_{w}^{t_2} = 2$.
    If $c_{x}^{t_3} \neq 2$, then between $t_2$ and $t_3$, one in $\{z_{l+1}', \ldots, z_d', y_1^{(2)}, z_2'', \ldots, z_{\ell}''\}$ is recolored by $2$ due to the way we choose best choice for $x$ at $t_2$.
    And since $c_{w}^{t_2} = 2$ and $\{w,y_2,\ldots,y_d\}$ is a clique, only $\{z_2'', \ldots, z_{\ell}''\}$ can be recolored by $2$ between $t_2$ and $t_3$.
    Note that before the recoloring of $y_1^{(2)}$, $w$ has color $2$, and at least one of $\{z_2'', \ldots, z_{\ell}''\}$ is recolored by $2$. 
    The color $2$ is used twice both in the colors of $N^-(y_1)$ before the recoloring of $y_1^{(2)}$ and in the new colors of the $d$ recolorings after the recoloring of $y_1^{(2)}$.
    Since the recoloring of $y_1^{(3)}$ is not the last recoloring of $y_1$, this contradicts that the recoloring of $y_1^{(2)}$ is tight by Observation~\ref{obs: tight recoloring property 2}.
    From above all, we must have $c_{x}^{t_3} = 2$.

    At moment $t_3$, the recoloring of $x$ is caused by $z_{l+1}''$.
    That is, after $t_3$, $z_{l+1}''$ is going to be recolored by $2$.
    Again by a similar argument as above, it contradicts that the recoloring of $y_1^{(2)}$ is tight.

\begin{figure}[t]
 \begin{center}
   \includegraphics[width=0.9\textwidth]{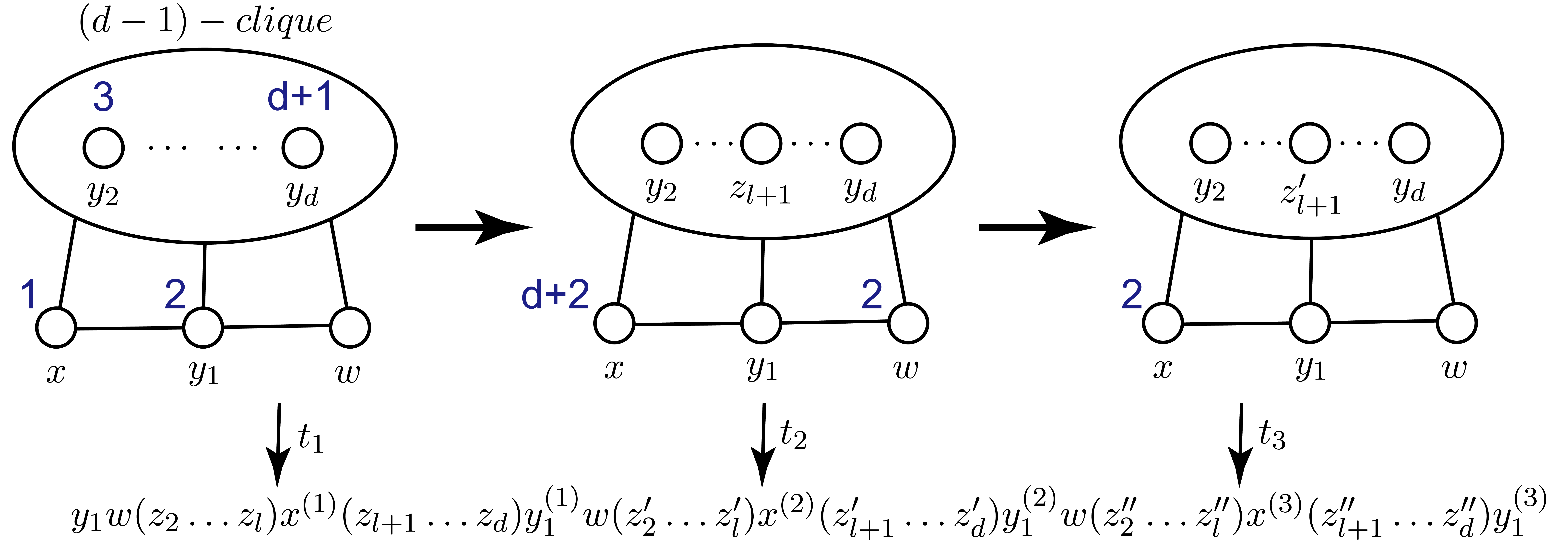}
\end{center}
    \caption{A sample of the pattern in the proof of statement 2. The blue numbers represent colors.}\label{fig: sequence}
\end{figure}

    Then, the recoloring of $x$ must be inserted exactly before $y_1$. Hence, the pattern

    \[
y_1w(z_2\ldots z_d)xy_1w(z_2'\ldots z_d') xy_1w(z_2''\ldots z_d'')xy_1w(z_2'''\ldots z_d''')xy_1,
    \]
    appears at least $c-16d^2(p+2)$ times in $\mathcal{S}_{|N^-[y_1]\cup \{x\}}$.

\begin{figure}[t]
\begin{center}
    \includegraphics[width=0.6\textwidth]{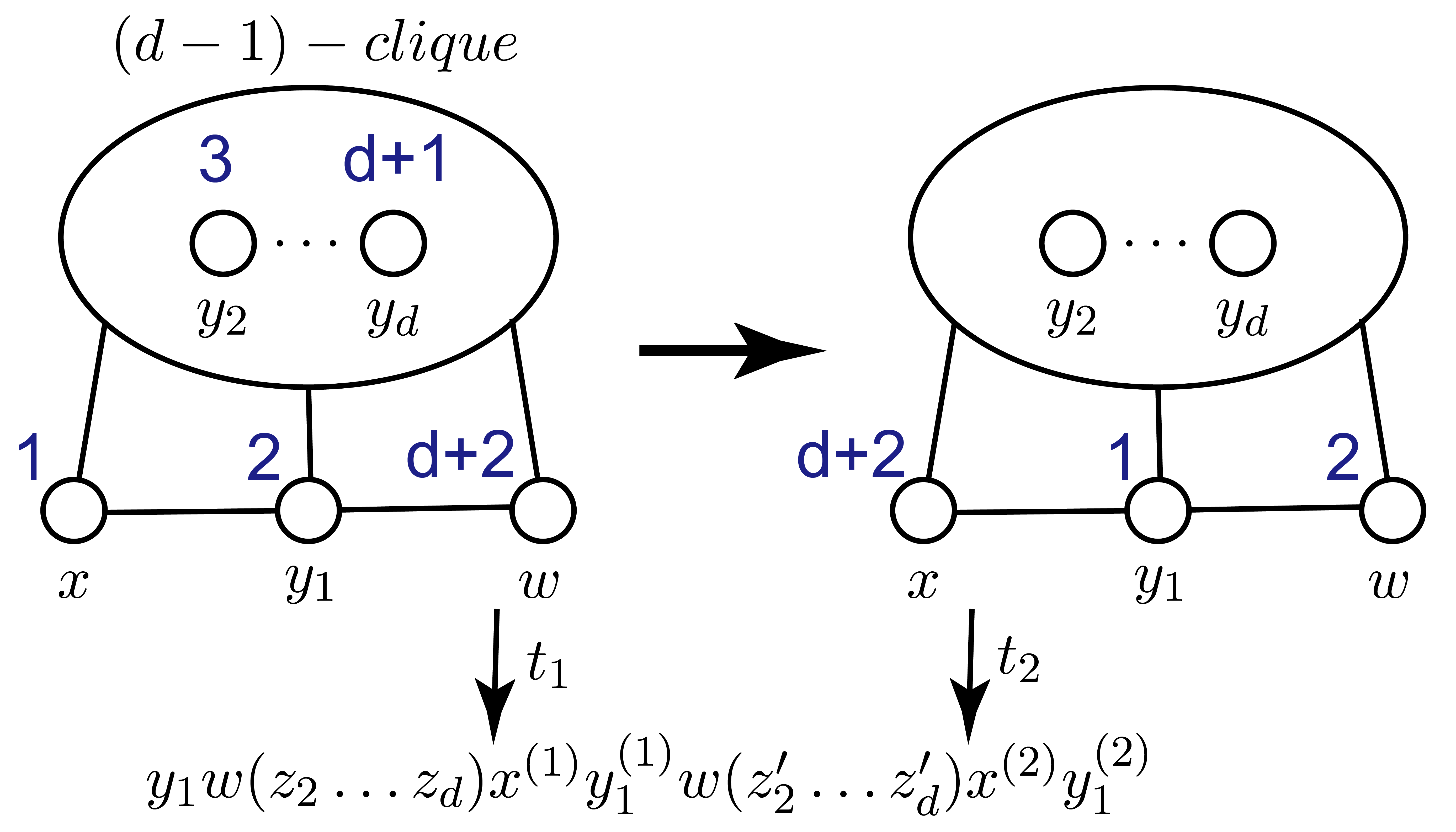}
\end{center}
    \caption{A sample of the pattern in the proof of statement 3. The blue numbers represent colors.}\label{fig: sequence 2}
\end{figure}

   
    \textbf{Proof of statement 3.}
    Now we focus on the (sub)-sequence
    \[
        y_1w(z_2\ldots z_d)x^{(1)}y_1^{(1)}w(z_2'\ldots z_d') x^{(2)}y_1^{(2)}.
    \]

    Similarly, let $t_i, i=1,2$ be the moment just before the recoloring of $x^{(i)}$.
    Without loss of generality, assume at moment $t_1$, $c^{t_1}_{x} = 1$ and $c^{t_1}_{y_i} = i+1, 1 \le i \le d$ and $x$ is going to be recolored by $d+2$.
    By Observation~\ref{obs: tight recoloring property 2} on tight recoloring of $x^{(1)}$, $\{z_2', \ldots, z_d'\}$ must use up all colors in $\{d+3, \ldots, 2d+1\}$.
    Since the recoloring of $x^{(1)}$ is caused by $y_1^{(1)}$, $y_1^{(1)}$ is going to be recolored by $1$.
    By Observation~\ref{obs: tight recoloring property 2} on tight recoloring of $y_1^{(1)}$, together with $c_{w}^{t_1}$, $\{z_2', \ldots, z_d'\}$ use up all colors in $\{d+2, d+3, \ldots, 2d+1\}$.
    Then, $c^{t_1}_{w}$ must be $d+2$.
    See Figure~\ref{fig: sequence 2} for an example.

    Hence the colors of $x, y_1$ and $w$ are looping.
    Here by looping, we mean, $x$ gets the color of $w$, then $y_1$ gets the color of $x$, and finally $w$ gets the color of $y_1$.
    The colors tuple of $x,y_1,w$ behaves like cyclic shifting one bit to the right.
    Similarly, it holds for the recolorings of $x^{(2)}$ and $x^{(3)}$.
    Then the recoloring of $x^{(1)}$ is rotating and hence there is at least $c-16d^2(p+2)$ rotating recoloring of $x$.
    Since $x$ is recolored at most $c+1$ times from Lemma~\ref{lemma: save inequality}, there are at most $16d^2(p+2)+1$ non-rotating recoloring of $x$.
\end{proof}

\vspace*{.2cm}
For a $(d-1)$-clique $X = \{x_1, \ldots, x_{d-1}\}$.
Assume $\mathcal{S}_{|X} = \{s_1', \ldots, s_m'\}$.
we say the recoloring $s_i'$ in $\mathcal{S}_{|X}$ is \textbf{naughty} if it satisfies the following properties.
    \begin{enumerate}
        \item There are three fixed colors that are neither in the colors of $X$ before $s_i'$, nor the new colors of $\{s_{i+1}', s_{i+2}', \ldots, s_{i+3d+4}'\}$.
        \item For every $j$ from $i+1$ to $i+3d-4$, $s_j'$ is not caused by $s_{j+1}'$.
    \end{enumerate}

Note that whether a recoloring of a vertex in $X$ is naughty only depends on the recoloring sequence $\mathcal{S}_{|X}$.




\noindent
\textbf{Proof of Theorem~\ref{thm: 2k+1 linear}}


Fix $c = 2^{18}d^7$.
Let $\mathcal{S}$ be the best choice recoloring sequence of $G$.
We are going to prove the following claim.

\begin{claim}\label{claim: naughty}
    For each vertex $v$, $v$ is recolored at most $c$ times in $\mathcal{S}$.
    And for each $(d-1)$-clique $X$, there are at most $(d-1)c-160d^3-1$ naughty recolorings.
\end{claim}

\noindent
\textbf{Proof of Claim~\ref{claim: naughty}:}
By induction on $|V(G)|$. It is trivial
when $|V(G)| \le d$.
Assume the claim is true for $|V(G)| \le n-1$, now consider when $|V(G)| = n$.
Let $v_1, \ldots, v_n$ be the perfect elimination ordering of $G$.
Then the claim holds for $G_{n-1}$.
That is, every vertex in $G_{n-1}$ is recolored at most $c$ times and for each $(d-1)$-clique $X$ in $G_{n-1}$, there are at most $(d-1)c-160d^3-1$ naughty recolorings.
We first prove that $v_n$ is recolored at most $c$ times.

If $v_n$ is recolored at least $c+1$ times, from Lemma~\ref{lemma: save inequality}, we have $|N^-(v_n)|=d$.
Write $x = v_n$ and $N^-(x) = \{ y_1, \ldots, y_d \}$ where $y_1$ has the largest order among all $v_i, 1 \le i \le d$.
From Lemma~\ref{lemma: save inequality}, we have that $y_i$ is recolored at least $c - d$ times.
By the property of perfect elimination ordering and Lemma~\ref{lemma: save inequality}, we have $|N^-(y_1)| = d$ and $y_i \in N^-(y_1), 2 \le i \le d$.
Write $N^-(y_1) = \{y_2, \ldots, y_d, w\}$.

By Lemma~\ref{lemma: refine sequence}, the patterns~(denoted by $\mathcal{P}$)
\[
    y_1w(z_2\ldots z_d)x^{(1)}y_1^{(1)}w(z_2'\ldots z_d') xy_1w(z_2''\ldots z_d'')xy_1w(z_2'''\ldots z_d''')xy_1
\]
appear at least $c-32d^2$ times in $\mathcal{S}_{|N^-[y_1]\cup \{x\}}$ where $z_j, z_j', z_j'',z_j''' \in \{y_2, \ldots, y_d\}, 2 \le j \le d$.
And at most $32d^2+1$ recolorings of $x$ is not rotating.
Then we have at least $c - 160d^2$ pattern $\mathcal{P}$ in which all the recolorings of $x$ are rotating.
Let $X = \{y_2, \ldots, y_d\}$, then $X$ is a $(d-1)$-clique from perfect elimination ordering.
Since the recoloring of $x$ is rotating and according to our analysis in proof of statement 3 in Lemma~\ref{lemma: refine sequence}, the colors of $x, y_1$ and $w$ are looping. 
Let $c_1,c_2,c_3$ denote the colors of $x,y_1,w$ before the recoloring of $x^{(1)}$.
Hence, the three colors $\{c_1,c_2,c_3\}$ never appear in $X$ which implies we have at least $(d-1)$ naughty recolorings~($z_2,\ldots,z_d$) in each $\mathcal{P}$.
And considering all such patterns, each naughty recoloring found in this way only counts once.
In total, we have at least $(d-1)c - 160d^3$ naughty recolorings in $\mathcal{S}_{|X}$ which contradicts our assumption.

It remains to prove that for each $(d-1)$-clique $X$ in $G$, there are at most $(d-1)c-160d^3-1$ naughty recolorings.
We only need to consider the cliques containing $x = v_n$.
If we have a $(d-1)$-clique $X$ containing $v_n$ and there are at least $(d-1)c-160d^3$ naughty recolorings, then  there are at least $c-160d^3$ naughty recolorings of $x$ due to that each vertex in $X\setminus \{x\}$ is recolored at most $c$ times.
Apply Lemma~\ref{lemma: refine sequence}, then we have that the pattern~(denoted by $\mathcal{P}$)
\[
    y_1w(z_2\ldots z_d)x^{(1)}y_1^{(1)}w(z_2'\ldots z_d') xy_1w(z_2''\ldots z_d'')xy_1w(z_2'''\ldots z_d''')xy_1
\]
appears at least $c-2^{10}d^5$ times in $\mathcal{S}_{|N^-[y_1]\cup \{x\}}$ where $z_j, z_j', z_j'',z_j''' \in \{y_2, \ldots, y_d\}, 2 \le j \le d$.
And at most $2^{13}d^5$ recolorings of $x$ is not rotating.
By Lemma~\ref{lemma: save inequality}, each $y_i, 1 \le i \le d$, is recolored at least $c - 2^{9}d^4$ times.
Again, apply Lemma~\ref{lemma: refine sequence} to $y_i, 1 \le i \le d$, at most $2^{14}d^6$ recolorings of $y_i$ is not rotating.
Then we have that pattern $\mathcal{P}$ appears at least $c - 2^{16}d^7$ times in which all the recolorings of $x$ and $y_i, 1 \le i \le d$, are rotating.

For the case $y_1 \in X$, at least $c - 2^{16}d^7$ recolorings of $x$ is caused by $y_1$.
Thus, at most $2^{16}d^7$ recolorings of $x$ is naughty~(since a naughty recoloring of $x$ can not be caused by $y_1$ and $x$ is recolored at most $c$ times).
Hence there are at most $(d-2)c + 2^{16}d^7$ naughty recolorings in $X$.
A contradiction arises since $(d-2)c + 2^{16}d^7 < (d-1)c - 160d^3$.

Then we consider the case when $y_1 \notin X$.
Assume $X = \{x, y_2, \ldots, y_d\} \setminus \{y_{\ell}\}$.
Now we consider the first recoloring of $x^{(1)}$ in $\mathcal{P}$.
Assume before this recoloring, the color of $x$ is $1$, the color of $y_i$ is $i+1$ for any $1 \le i \le d$, and $x$ is going to be recolored by $d+2$ which is the color of $w$ according to our analysis in proof of statement 3 in Lemma~\ref{lemma: refine sequence}.
Here we have the same configuration as Figure~\ref{fig: sequence 2}.
Whatever $y_{\ell}$ is, the $3d-4$ recolorings after $x^{(1)}$ in $\mathcal{S}_{|X}$ contains three recolorings of $x$.
Since the recoloring of $x^{(1)}$ is rotating, the colors $d+2$ and $2$ are used in $X$ in the $3d-4$ steps after $x^{(1)}$ in $\mathcal{S}_{|X}$.

If the recoloring of $x^{(1)}$ is naughty, note that before this step, $X$ has colors $\{1,3,4,\ldots,d+1\}\setminus\{\ell+1\}$.
Then there are three colors in $\{\ell+1\}\cup \{d+3, d+4, \ldots, 2d+1\}$ that are not used in $X$ in the $3d-4$ recolorings after $x^{(1)}$ in $\mathcal{S}_{|X}$.
Since the recoloring of $y_1^{(1)}$ is tight, by Observation~\ref{obs: tight recoloring property 2}, $z_2', \ldots, z_d'$ use up all colors in $\{d+3, \ldots, 2d+1\}$.
The only possibility is that there are $2 \le i_1, i_2 \le d$ such that $z_{i_1}' = z_{i_2}' = y_{\ell}$.
The three unused colors are the color $\{\ell+1\}$ and the colors of $y_{\ell}$ in the two recolorings $z_{i_1}', z_{i_2}'$.
Then for every naughty recoloring of $x^{(1)}$ in $\mathcal{P}$, we can find two occurrences of recoloring of $y_{\ell}$ in $\{z_2', \ldots, z_d'\}$. And each occurrence of $y_{\ell}$ is counted at most once. 
Then there is at most $c/2 + 2^{16}d^7$ such naughty recolorings of $x$ since there are at most $c$ recoloring of $y_{\ell}$.

Therefore, there is at most $(d-2)c + c/2 + 2^{16}d^7$ naughty recolorings in $\mathcal{S}_{|X}$.
This leads to a contradiction since $(d-2)c + c/2 + 2^{16}d^7 < (d-1)c - 160d^3$.
\hfill $\blacksquare$\par

Given any $d$-degenerate and chordal graph $G$ and any two $t$-coloring~($t=2d+1$) $\alpha$ and $\beta$ of $G$.
We can first obtain a perfect elimination ordering of $G$, then apply the best choice algorithm to obtain the best choice recoloring sequence from $\alpha$ to $\beta$.
By Claim~\ref{claim: naughty}, each vertex is recolored at most $c$ times.
\hfill $\square$\par

\section{Discussion}\label{sec: discussion}

A natural problem arises here: is the bound in Theorem~\ref{thm: 2k+1 linear} and Corollary~\ref{cor: 2k+1 linear} tight?
The question was also asked by Bartier, Bousquet and Heinrich in~\cite{bartier2021recoloring}.
Let $\mu(k)$~(resp. $\nu(k)$) be the smallest integer such that for any graph $G$ of treewidth at most $k$~(resp. $k$-degenerate and chordal), the $t$-recoloring diameter of $G$ is at most linear if $t \ge \mu(k)$~(resp. $t \ge \nu(k)$).
Theorem~\ref{thm: 2k+1 linear} and Corollary~\ref{cor: 2k+1 linear} show that $\mu(k), \nu(k) \le 2k+1$.

In~\cite{bonamy2014reconfiguration}, Bonamy, Johnson, Lignos, Patel and Paulusma~constructed a graph $G$ of treewidth $k$ and chordal whose $(k+2)$-recoloring diameter is at least $\Omega(|V(G)|^2)$.
This shows $\mu(k),\nu(k) \ge k+3$.
However, when $t \ge k+3$, the example they provided has at most linear $t$-recoloring diameter. Thus, their construction does not improve the lower bound of $\mu(k)$ and $\nu(k)$ any further.

We believe that $\mu(k) = \nu(k) = k+3$, but new ideas may be required to prove it.

\section*{Acknowledgement}

The authors would like to thank the referee for their careful reading and valuable comments which help to improve the presentation of this paper a lot.
M. Lu is supported by the National Natural Science Foundation of China~(Grant 12571372).


\begin{thebibliography}{10}

\bibitem{bartier2021recoloring}
V.~Bartier, N.~Bousquet, and M.~Heinrich.
\newblock Recoloring graphs of treewidth 2.
\newblock \emph{Discrete Mathematics}, 344(12):112553, 2021.

\bibitem{bonamy2013recoloring}
M.~Bonamy and N.~Bousquet.
\newblock Recoloring bounded treewidth graphs.
\newblock \emph{Electronic Notes in Discrete Mathematics}, 44:257--262, 2013.

\bibitem{bonamy2018recoloring}
M.~Bonamy and N.~Bousquet.
\newblock Recoloring graphs via tree decompositions.
\newblock \emph{European Journal of Combinatorics}, 69:200--213, 2018.

\bibitem{bonamy2014reconfiguration}
M.~Bonamy, M.~Johnson, I.~Lignos, V.~Patel, and D.~Paulusma.
\newblock Reconfiguration graphs for vertex colourings of chordal and chordal bipartite graphs.
\newblock \emph{Journal of Combinatorial Optimization}, 27(1):132--143, 2014.

\bibitem{bousquet2019linear}
N.~Bousquet and V.~Bartier.
\newblock {Linear Transformations Between Colorings in Chordal Graphs}.
\newblock In M.~A. Bender, O.~Svensson, and G.~Herman, editors, \emph{27th Annual European Symposium on Algorithms (ESA 2019)}, volume 144 of \emph{Leibniz International Proceedings in Informatics (LIPIcs)}, pages 24:1--24:15. Schloss Dagstuhl -- Leibniz-Zentrum f{\"u}r Informatik, Dagstuhl, Germany, 2019.

\bibitem{bousquet2022polynomial}
N.~Bousquet and M.~Heinrich.
\newblock A polynomial version of cereceda's conjecture.
\newblock \emph{Journal of Combinatorial Theory, Series B}, 155:1--16, 2022.

\bibitem{bousquet2016fast}
N.~Bousquet and G.~Perarnau.
\newblock Fast recoloring of sparse graphs.
\newblock \emph{European Journal of Combinatorics}, 52:1--11, 2016.

\bibitem{CAMBIE2024103798}
S.~Cambie, W.~{Cames van Batenburg}, and D.~W. Cranston.
\newblock Optimally reconfiguring list and correspondence colourings.
\newblock \emph{European Journal of Combinatorics}, 115:103798, 2024.

\bibitem{2024arXiv241219695C}
S.~{Cambie}, W.~{Cames van Batenburg}, and D.~W. {Cranston}.
\newblock {Sharp Bounds on Lengths of Linear Recolouring Sequences}.
\newblock \emph{arXiv e-prints}, arXiv:2412.19695, December 2024.

\bibitem{2025arXiv250508020C}
S.~{Cambie}, W.~{Cames van Batenburg}, D.~W. {Cranston}, J.~{van den Heuvel}, and R.~J. {Kang}.
\newblock {Reconfiguration of List Colourings}.
\newblock \emph{arXiv e-prints}, arXiv:2505.08020, May 2025.

\bibitem{cereceda2007mixing}
L.~Cereceda.
\newblock \emph{Mixing graph colourings}.
\newblock Ph.D. thesis, London School of Economics and Political Science, 2007.

\bibitem{Dyer2006}
M.~Dyer, A.~D. Flaxman, A.~M. Frieze, and E.~Vigoda.
\newblock Randomly coloring sparse random graphs with fewer colors than the maximum degree.
\newblock \emph{Random Structures \& Algorithms}, 29(4):450--465, 2006.

\bibitem{feghali2020reconfiguration}
C.~Feghali and J.~Fiala.
\newblock Reconfiguration graph for vertex colourings of weakly chordal graphs.
\newblock \emph{Discrete Mathematics}, 343(3):111733, 2020.

\bibitem{feghali2022mixing}
C.~Feghali and O.~Merkel.
\newblock Mixing colourings in {$2K_2$}-free graphs.
\newblock \emph{Discrete Mathematics}, 345(11):113108, 2022.

\bibitem{jerrum1995very}
M.~Jerrum.
\newblock A very simple algorithm for estimating the number of {$k$}-colorings of a low-degree graph.
\newblock \emph{Random Structures \& Algorithms}, 7(2):157--165, 1995.

\bibitem{2024arXiv240919368L}
H.~{Lei}, Y.~{Ma}, Z.~{Miao}, Y.~{Shi}, and S.~{Wang}.
\newblock {Reconfiguration graphs for vertex colorings of $P_5$-free graphs}.
\newblock \emph{Annals of Applied Mathematics}, 40(4):394--411, 2024.

\end{thebibliography}

\end{document}